\RequirePackage{lineno}
\documentclass[12pt,a4paper,leqno]{amsart}
\newcounter{minutes}
\divide\time by 60
\newcounter{hours}
\multiply\time by 60
\addtocounter{minutes}{-\time}
\usepackage{amssymb}
\usepackage{float}
\usepackage{here}
\usepackage{graphicx}

\date{}
\newfont{\cyrilic}{wncyr10 scaled 1000}
\newpage
\title{Norm inequalities for vector functions}
\author { B. A. Bhayo, V. Bo\v zin, D. Kalaj, M. Vuorinen}
\newcommand{\comment}[1]{}

\swapnumbers
\theoremstyle{plain}

\newtheorem{theorem}[equation]{Theorem}
\newtheorem{lemma}[equation]{Lemma}
\newtheorem{definition}[equation]{Definition}

\newtheorem{corollary}[equation]{Corollary}

\newtheorem{remark}[equation]{Remark}

\newtheorem{subsec}[equation]{}

\numberwithin{equation}{section}

\begin{document}
%%%%%KAUNIS K  \K %%%%%%%%%%%%%%
\font\fFt=eusm10 %scaled 1200
\font\fFa=eusm7  %scaled 1200
\font\fFp=eusm5  %scaled 1200
\def\K{\mathchoice
 %%%displaystyle
{\hbox{\,\fFt K}}
%%%%textstyle
{\hbox{\,\fFt K}}
%%%scriptstyle
{\hbox{\,\fFa K}}
%%%%scriptscriptstyle
{\hbox{\,\fFp K}}}
\maketitle
\begin{center}
{\tiny \texttt{File:~\jobname .tex,
          printed: \number\year-\number\month-\number\day,
          \thehours.\ifnum\theminutes<10{0}\fi\theminutes} }
\end{center}

%\begin{center}
{\bf Abstract.}
%\end{center}
We study vector functions of ${\mathbb R}^n$ into itself, which are of the form
$x \mapsto g(|x|)x\,,$ where $g : (0,\infty) \to (0,\infty) $ is a
continuous function and call these radial functions. In the case when
$g(t) = t^c$ for some $c \in {\mathbb R}\,,$ we find
upper bounds for the distance of image points under such a radial function. Some
of our results refine recent results of L. Maligranda and S. Dragomir.
In particular, we study quasiconformal
mappings of this simple type and obtain norm inequalities for such mappings.


\vspace{.5cm}
%%%%%%%%%%%%
\section{Introduction}

In 2006 L. Maligranda \cite{m} studied the following function
\begin{equation} \label{defpangdist}
\alpha_p(x,y) = ||x|^{p-1} x - |y|^{p-1}y|\,,  p \in {\mathbb R} \, ,
\end{equation}
for $x,y \in  {\mathbb R}^n \setminus \{ 0 \}\,,$ termed the $p-$angular
distance between $x$ and $y\,.$ It is clear that $\alpha_p$ satisfies the
triangle inequality and thus it defines a metric.
Note that  $\alpha_0(x,y)$ equals $2 \sin(\omega/2)$
where $\omega \in [0,\pi]$ is the angle between the segments $[0,x]$ and  $[0,y]\,.$
He proved in \cite[Theorem 2]{m} the following theorem in the context of normed spaces.

\begin{theorem} \label{Mthm}
$$\alpha_p(x,y)\leq \left\{\begin{array}{lll} (2-p)\displaystyle\frac{|x-y| \max\{|x|^{p},|y|^{p}\}}
                      {(\max\{|x|,|y|\})} \quad
                      {\rm if}\, \, p\in (-\infty,0) \, \, {\rm and}\, x,y\neq 0;
                    \\   \\(2-p)\displaystyle\frac{|x-y|}
                      {(\max\{|x|,|y|\})^{1-p}} \quad
                      {\rm if}\,\, p\in[0,1] \, \, {\rm and}\,\, x,y\neq 0;\\
                     \\
                       p \,(\max\{|x|,|y|\})^{p-1}|x-y| \quad {\rm if}
                       \,p \, \in(1,\infty)
                     .\end{array}\right.$$
\end{theorem}

Soon thereafter, in 2009, S. Dragomir \cite[Theorem 1]{d} refined this result
and gave the following upper bound for the $p$-angular distance for nonzero
vectors $x,y\,.$

\begin{theorem} \label{Dthm}
$$\alpha_p(x,y)\leq \left\{\begin{array}{llllll} |x-y|(\max\{|x|,|y|\})^{p-1}+
                 \left||x|^{p-1}-|y|^{p-1}\right|\min\{|x|,|y|\}\quad
                  {\rm if} \, \, p\in(1,\infty)\,;\\
                 \\
                  \displaystyle\frac{|x-y|}{(\min\{|x|,|y|\})^{1-p}}+
                 \left||x|^{1-p}-|y|^{1-p}\right|
                 \min\left\{\frac{|x|^p}{|y|^{1-p}},\frac{|y|^p}{|x|^{1-p}}\right\} \quad
                 {\rm if} \, \, p\in[0,1]\,; \\
    \\
     \displaystyle\frac{|x-y|}{(\min\{|x|,|y|\})^{1-p}} +
    \displaystyle\frac{||x|^{1-p}-|y|^{1-p}|}{(\max\{|x|^{-p}|y|^{1-p},|y|^{-p} |x|^{1-p}\})} \quad {\rm if} \, \, p \in (-\infty,0)\,.
                 \end{array}\right.$$

%                 Here $\alpha_p[x,y]=||x|^{p-1}x-|y|^{p-1}y|$.
\end{theorem}

Generalizations for operators were discussed very recently in \cite{dfm}.
%Various other results can be found in \cite{dm}.
For general information about norm inequalities see \cite[Chapter XVIII]{mpf}.

Studying sharp constants connected to the $p$-Laplace operator J. Bystr\"om
\cite[Lemma 3.3]{by} proved in 2005 the following result.

\begin{theorem}\label{Bthm} For $p\in(0,1)$
and $x,y\in \mathbb{R}^n$, we have
$$\alpha_p(x,y) \leq 2^{1-p}|x-y|^{p}$$
with equality for $x=-y\,.$
\end{theorem}

In this paper we study a two exponent variant of the function $x \mapsto |x|^{p-1}x$
defined for $a, b >0, x \in \mathbb{R}^n\,,$
\begin{equation} \label{myfdef}
{\mathcal A}_{a,b}(x)=\left\{\begin{array}{ll}|x|^{a-1}x\quad if\; |x|<1\\
                   |x|^{b-1}x\quad if\; |x|\ge 1.\end{array}\right.
\end{equation}
This function, like its one exponent version (the special case $a=b$),
defines a quasiconformal mapping and it has been used
in many examples to illuminate various properties of these maps \cite[p.49]{v}. For instance,
if $a \in (0,1)$ the function $  {\mathcal A}_{a,b}$ is H\"older-continuous at the origin.

We prove that the change of distance under this function is maximal in the radial
direction, up to a constant, in the sense of the next theorem
(observe that the points $x$ and $z$ are on the same ray). Note that the result is sharp
for $a \to 1\,.$ This result is natural to expect, but the proof is somewhat
involved. For brevity we write ${\mathcal A}={\mathcal A}_{a,b}$ if $ 0 < a \le 1 \le b\,.$

\begin{theorem} \label{kal}
Let $0<a\le 1\le b$ and
$$C(a,b) = \sup_{|x|\le|y|}Q(x,y),$$ where
$$Q(x,y)=\frac{|{\mathcal A}(x)-{\mathcal A}(y)|}
{|{\mathcal A}(x)-{\mathcal A}(z)|}\,,\quad  x,y \in {\mathbb R}^n \setminus \{ 0 \}\,
\text{ with }\; x\neq y\,,$$ and
$$z=\frac{x}{|x|}(|x|+|x-y|).$$
Then $$C(a,b)=\frac{2}{3^a-1}\text{ and }\lim_{a\to 1} C(a,b)=1.$$
\end{theorem}

Because ${\mathcal A}_{a,b}$ agrees with $x \mapsto  |x|^{a-1}x$ in ${\mathbb B}^n\,,$
we can compare
Theorem \ref{kal} to Theorems \ref{Mthm}, \ref{Dthm}, and \ref{Bthm}.
%In conclusion we
%see that sometimes Theorem \ref{kal} gives the best upper bound.
%{\Large \bf More comments about these results will be inserted here when the comparison tests
%will be ready.}
We also have the following upper bound for $\alpha_p\,:$

\begin{theorem} \label{mymaj} For all $x,y \in {\mathbb R}^n$ and $p \in (0,1)$
\begin{equation}
\alpha_p(x,y) \le |{\mathcal A}_{p, 1/p}(x) - {\mathcal A}_{p, 1/p}(y)  | \,,
\end{equation}
and furthermore, if $|x|\le |y|$, we have also
\begin{equation}
\alpha_p(x,y) \le |{\mathcal A}_{p, 1/p}(x) - {\mathcal A}_{p, 1/p}(y)  |
\le
\frac{2}{3^{p}-1}  |{\mathcal A}_{p, 1/p}(x) - {\mathcal A}_{p, 1/p}(z)  |
\end{equation}
where $z$ is as in Theorem \ref{kal}.
\end{theorem}

For a systematic comparison of the above results, see Section 5 where it is shown
that sometimes the bound in Theorem \ref{mymaj} is better than the other bounds
in Theorems \ref{Mthm}, \ref{Dthm}, \ref{Bthm}.

We also discuss some properties of the distortion function $\varphi_K(r)$
associated with the quasiconformal Schwarz lemma, see \cite{lv}.

\bigskip

{\sc Acknowledgments.} The first author is indebted to the Graduate
School of Mathematical Analysis and its Applications for support.
He also wishes to acknowledge the expert help of
Dr. H. Ruskeep\"a\"a in the use of the Mathematica$^{\tiny \textregistered}$ software \cite{ru}.
The fourth author was, in part, supported by the Academy of Finland, Project 2600066611.

\section{Preliminary results}

We prove here some inequalities for elementary functions that will
be applied in later sections. These inequalities deal with the
logarithm and some of them may be new results. Note also in the
paper \cite{kmv} some elementary Bernoulli type inequalities were
proved and used as a key tool. We use the notation sh, ch, th, arsh,
arch and arth to denote the hyperbolic sine, cosine,
tangent and their inverse functions, respectively.

As well-known, conformal invariants of geometric function theory
are on one hand closely linked with function theoretic extremal
problems and on the other hand with special functions such as
complete elliptic integrals, elliptic functions and hypergeometric
functions. The connection between conformal invariants and special
functions is provided by conformal maps which can be applied to express
maps of quadrilaterals and ring domains onto canonical ring domains such
as a rectangle and an annulus.

For example, the quasiconformal version of the Schwarz lemma says that for a
$K$-quasiconformal map of the unit disk $\mathbb{B}^2$ onto itself
keeping $0$ fixed, we have for all $z \in \mathbb{B}^2$
the sharp bound \cite[p. 64]{lv}
\begin{equation}
|f(z)| \le \varphi_K(|z|)\,, \quad \varphi_K(r) = \mu^{-1}(\mu(r)/K)
\end{equation}
where
$\mu:(0,1)\longrightarrow(0,\infty)$ is a decreasing homeomorphism defined by
\begin{equation}
\label{dec_hom} \mu(r)=\frac\pi 2\,\frac{{\K}(r')}{{\K}(r)}, \quad {\K}(r)=\int_0^1\frac{dx}{\sqrt{(1-x^2)(1-r^2x^2)}}\, ,
\end{equation}
and where ${\K}(r)$ is Legendre's complete elliptic integral of the
first kind and $ r'=\sqrt{1-r^2},$ for all $r\in(0,1)$.
The function $\varphi_K(r)$ has numerous applications to quasiconformal mapping theory,
see \cite{lv,k,avvb}, which motivates the study of its properties. One of the challenges
is to find bounds, in the range $(0,1)$, and yet asymptotically sharp when $K \to 1\,.$
For instance, the change of hyperbolic distances under $K$-quasiconformal mappings of the
unit disk onto itself can be estimated in terms of the function $\varphi_K\,,$ see
\cite{avvb,lv}.

\begin{lemma}\label{2C} The following functions are monotone increasing
from $(0,\infty)$ onto $(1,\infty)$;

\noindent $(1)\qquad f(x)=\displaystyle\frac{(1+x)\log(1+x)}{x}\,,\quad  (2)\qquad g(x)=\displaystyle\frac{x}{\log(1+x)}\,,$\\
$(3)$ For a fixed $t\in(0,1)$, the function $h(K)=K(1-t^{2/K})$ is monotone
increasing on $(1,\infty)$.
\end{lemma}

\begin{proof} For the proof of (1) see \cite[p. 7]{kmv}. For (2), we get
$$g^{'}(x)=\frac{1}{\log(1+x)}-\frac{x}{(1+x)(\log(1+x))^2}
=\frac{(1+x)\log(1+x)-x}{(1+x)(\log(1+x))^2},$$
and $g^{'}(x)>0$ by (1). Moreover,
$g$ tends to $1$ and $\infty$ when $x$ tends $0$ and $\infty$.
Proof of $(3)$ follows easily because $x\mapsto(1-a^{x})/x$ is decreasing on $(0,1)$ for each $a\in (0,1)$ \cite[1.58(3)]{avvb}.
\end{proof}

\begin{corollary}\label{2e} For a fixed $x\in(0,1)$,
the following functions, (1) $f(a)=(1+ax)^{1/a}$,
(2) $g(a)=(\log(1+x^a))^{1/a}$ are decreasing
and increasing on $(1,\infty)$, respectively.
(3) The following inequality holds for  $x\geq 0$ and $a\in[0,1]$,
$$\log(1+x^a)\leq \max\{\log(1+x),\log^a(1+x)\}.$$
\end{corollary}

\begin{lemma}\label{2e0} For $K>1\,,r\in(0,1),\,u={\rm arch}(1/r)/K$, the following functions
\begin{enumerate}
\item $f(K)=r\,{\rm arth}(1/{\rm ch(u)})
{\rm sh}(u),$
\item $g(K)=rK\,{\rm arth}(1/{\rm ch(u)})
{\rm sh}(u)$
\end{enumerate}
are strictly decreasing and increasing, respectively. Moreover,
both functions tend to $\sqrt{1-r^2}\,{\rm arth}(r)$ when $K$ tends to $1$.
\end{lemma}

\begin{proof} Differentiating $f$ with respect to $K$ we get

$$f^{'}(K)=-\frac{r\,\textrm{arth}(1/r)}{K^2}\left(\textrm{arth}
\left(\frac{1}{\textrm{ch}(u}\right)
{\rm ch}(u)-1\right)\leq 0,$$
$$g^{'}(K)=r\left(\textrm{ch}\left(\frac{1}{r}\right)
\left(1-\textrm{arth}\left(\frac{1}{\textrm{ch}(u)}\right)
+K\,\textrm {arth}\left(\frac{1}{\textrm{ch}(u)}\right)
\textrm{sh}(u)\right)\right)\geq 0,$$
respectively. We obtain
$$f(1)=g(1)=r\,\textrm{arth}(r)\sqrt{({\rm ch}({\rm arch}(1/r)))^2-1}=\sqrt{1-r^2}\,\textrm{arth}(r).$$
\end{proof}

\begin{lemma}\label{2ll} (1) For a fixed $t>0$, the following function is monotone
increasing in $K>1$. Moreover, for $t=t_0=(e-1)/(e+1)$, the function is increasing from
$(1,\infty)$ onto $(m_1,1)$,
$$f(K)=\frac{K-\log(1/t)}{t^{2/K}(K+\log(1/t))}
,\; m_1=\frac{1+\log t_0}{t^2_0(1-\log t_0)}\approx 0.6027..$$
(2) The following function is monotone increasing
 from $(1,\infty)$ onto $(m_2,1)$,
 $$g(K)=\frac{t^{1/K}_0\log(1/t^2_0)}{K(1-t^{2/K}_0)},
 \; m_2=\frac{2t_0\log t_0}{t^2_0-1}\approx 0.9072..\;.$$
\end{lemma}

\begin{proof} Differentiating $f$ with respect to $K$ we get
\begin{eqnarray*}
f^{'}(K)&=&\frac{t^{-2/K}}{K+\log(1/t)}-\frac{t^{-2/K}(K-\log(1/t))}{(K+\log(1/t))^2}
               +\frac{2t^{-2/K}(K-\log(1/t))\log t}{K^2(K+\log(1/t))^2}\\
           	 &=&\frac{2t^{-2/K}(K^2\log(1/t)+K^2\log t-(\log(1/t))^2\log t)}{K^2(K+\log(1/t))^2}\\
        &=&\frac{2t^{-2/K}(\log(1/t))^3}{K^2(K+\log(1/t))^2}>0.
\end{eqnarray*}
For $t=t_0$, $f$ tends to $m_1$ and $1$ when $K$ tends to $1$ and $\infty$,
respectively.\\
For the proof of (2), we differentiate $g$ with respect to $K$ and get,
\begin{eqnarray*}
 g^{'}(K)&=&-\frac{t^{1/K}_0\log(1/t^2_0)}{K^2(1-t^{2/K}_0)}
            -\frac{2t^{3/K}_0\log(1/t^2_0)\log t_0}{K^3(1-t^{2/K}_0)^2}
             -\frac{t^{1/K}_0\log(1/t^2_0)\log t_0}{K^3(1-t^{2/K}_0)}\\
         &=&t^{1/K}_0\log(1/t^2_0)(-K(1-t^{2/K}_0)-(1+t^{2/K}_0)\log t_0)/(K^3(1-t^{2/K}_0)^2)\\
         &=&t^{1/K}_0\log(1/t^2_0)(t^{2/K}_0(1+\log(1/t_0))-(K-\log (1/t_0)))/(K^3(1-t^{2/K}_0)^2)\\
          &=&\frac{t^{3/K}_0\log(1/t^2_0)(1+\log(1/t_0))}{K^3(1-t^{2/K}_0)^2}\left(1-
          \frac{K-\log (1/t_0)}{t^{2/K}_0(1+\log(1/t_0))}\right)>0
\end{eqnarray*}
by (1). We can see that $g$ tends to $m_2$ and
$1$ when $K$ tends to $1$ and $\infty$, respectively. This completes the proof.
\end{proof}
%%%%%%%%%%%
%%%%%%%%%%%%

\begin{lemma} The following inequality holds for
$K\geq 1 $ and $t\in[t_0,1),t_0=(e-1)/(e+1)$

\begin{equation}\label{1c}
\log\left(\frac{1+t^{1/K}}{1-t^{1/K}}\right)\leq
K \log\left(\frac{1+t}{1-t}\right).
\end{equation}
\end{lemma}

\begin{proof} Write $h(t)=K\textrm{arth}(t)-\textrm{arth}(t^{1/K})$.
Differentiating $h$ with respect to $t$ we get,
\begin{eqnarray*}
h^{'}(t)&=&\frac{K}{1-t^2}-\frac{t^{1/K-1}}{K(1-t^{2/K})}
         =\frac{K^2t(1-t^{2/K})-t^{1/K}(1-t^2)}{tK(1-t^2)(1-t^{2/K})}\\
        &\geq&\frac{Kt(1-t^2)-t^{1/K}(1-t^2)}{tK(1-t^2)(1-t^{2/K})}
       =\frac{Kt-t^{1/K}}{Kt(1-t^{2/K})}\geq 0.\\
\end{eqnarray*}
The first inequality holds by Lemma \ref{2C}(3) and the second one holds
when $Kt\geq t^{1/K}\Leftrightarrow t\geq (1/K)^{K/(K-1)}=c_1(K)$. It is easy to see by
Lemma \ref{2C}(1) that $c_1(K)$ is decreasing in $(1,\infty)$. We see that $c_1(K)\to 1/e\approx 0.3679..$ and $0$ when $K\to 1$ and $\infty$ respectively, hence $h(t)$ is increasing in $t\geq 1/e$.\\
We can see that $h(t_0)=K(1-2\,\textrm{arth}(t^{1/K}_0)/K)/2$.
Now it is enough to prove that
$f(K)=2\,\textrm{arth}(t^{1/K}_0)/K<1$.
Differentiating $f$ with respect to $K$ we get
\begin{eqnarray*}
f^{'}(K)&=&\frac{-2\,\textrm{arth}(t^{1/K}_0)}{K^2}-
           \frac{2t^{1/K}_0\log(t_0)}{K^3(1-t^{2/K}_0)}\\
           &=&2(-K(1-t^{2/K}_0)\,\textrm{arth}(t^{1/K}_0)-
           t^{1/K}_0\log(t_0))/(K^3(1-t^{2/K}_0))\\
           &\leq&2(-K(1-t^{2/K}_0)\,\textrm{arth}(t_0)+t^{1/K}_0\log(1/t_0))/
           (K^3(1-t^{2/K}_0))\\
           &=&2(-(K(1-t^{2/K}_0)/2)\log\left(\frac{1+t_0}{1-t_0}\right)
           +t^{1/K}_0\log(1/t_0))/(K^3(1-t^{2/K}_0))\\
           &=&(t^{1/K}_0\log(1/t^2_0)-K(1-t^{2/K}_0))
          /(K^3(1-t^{2/K}_0))\\
           &=&\frac{1}{K^2}\left(\frac{t^{1/K}_0\log(1/t^{2}_0)}{K(1-t^{2/K}_0)}-1\right)<0
\end{eqnarray*}
by Lemma $\ref{2ll}(2)$, hence $f$ is a monotone decreasing function
from $(1,K)$ onto $(0,1/2)$. This implies the proof.
\end{proof}

\begin{lemma} The following inequality holds for
$K\geq 1 $ and $t\in(0,t_0],t_0=(e-1)/(e+1)$
\begin{equation}\label{1a}
\log\left(\frac{1+t^{1/K}}{1-t^{1/K}}\right)\leq
K \left(\log\left(\frac{1+t}{1-t}\right)\right)^{1/K}.
\end{equation}
\end{lemma}

\begin{proof}
Write
$$F(t)=K-\displaystyle\frac{\log((1+t^{1/K})/(1-t^{1/K}))}
{\left(\log(1+t)/(1-t)\right)^{1/K}}\;.$$
For the proof of (\ref{1a}) we show that $F(t)$ is decreasing in $t$ and $F(t_0)\geq 0$.
Differentiating $F$ with respect to $t$ we get,

$$F^{'}(t)=\frac{\log\left(\frac{1+t}{1-t}\right)^{(K-1)/K}\left(2t^{1/K}(t^2-1)\log\left(\frac{1+t}{1-t}\right)-2t(t^{2/K}-1)\log\left(\frac{1+t^{1/K}}{1-t^{1/K}}\right)\right)}{Kt(t^2-1)(t^{2/K}-1)}\;.$$

Now we show that
$$t(t^{2/K}-1)\log\left(\frac{1+t^{1/K}}{1-t^{1/K}}\right)\geq t^{1/K}(t^2-1)
\log\left(\frac{1+t}{1-t}\right).\qquad (\ast)$$
For the proof of $(\ast)$, it is enough to prove that
$t(t^{2/K}-1)\geq t^{1/K}(t^2-1).$
We get
\begin{eqnarray*}
t(t^{2/K}-1)-t^{1/K}(t^2-1)&=&(t^{1/K+1}+t)(t^{1/K}-1)-(t^{1/K+1}+t^{1/K})(t-1)\\
                           &=&t^{1/K+1/K+1}+t^{1/K}-t-t^{1/K+1+1}\\%&=&t^{1/K+1/K+1}-t^{1/K+1}+t^{1/K+1}-t-
                           &=&t^{1/K}(t^{1/K+1}+1)-t(t^{1/K+1}+1)\\%t^{1/K+1+1}+t^{1/K+1}-t^{1/K+1}+t^{1/K}\\
                           &=&(t^{1/K+1}+1)(t^{1/K}-t)\geq 0,
\end{eqnarray*}
this implies that $F(t)$ is decreasing in $t$. Now we prove that $F(t_0)$
is positive as a function of $K$.
We write
$$f(K)=K-\log\left(\frac{1+t^{1/K}_0}{1-t^{1/K}_0}\right)
=K-2\, \textrm{arth} (t^{1/K}_0)=F(t_0)\;.$$
Differentiating $f$ with respect to $K$ we get
$$f^{'}(K)=1+\frac{2t^{1/K}_0\log(t_0)}{K^2(1-t^{2/K}_0)}
\geq 1-\frac{t^{1/K}_0\log(1/t^2_0)}{K^2(1-t^{2/K}_0)}>0$$
by Lemma \ref{2ll}(2),
hence $f$ is increasing in $K$. This implies the proof.
\end{proof}

\begin{corollary} The following inequality holds for
$K\geq 1 $ and $t\in [0,1)$
\begin{equation} \label{1ccc}
\log \left( \frac{1+t^{1/K}}{1-t^{1/K}} \right) \leq K
 \max \left\{ \left(  \log \left(\frac{1+t}{1-t}\right) \right)^{1/K},
  \, \, \log \left(\frac{1+t}{1-t}\right) \right\} .
\end{equation}
\end{corollary}

\begin{proof} The proof follows easily from
inequalities (\ref{1c}) and (\ref{1a}).
\end{proof}

The next function tells us how the hyperbolic distances from the
origin are changed under the radial selfmapping of the the unit disk,
$z \mapsto |z|^{1/K -1} z, K>1,$  which is the restriction of
${\mathcal A}_{1/K,1/K}(z)$ to the unit disk. See also \cite{bv}.

\begin{theorem}\label{1dd} The following inequality holds for $K\geq 1$, $|z|<1$;
\begin{equation}\label{1dD}
\rho(0,{\mathcal A}_{1/K,K}(z))\leq K\max\{\rho(0,|z|),\rho^{1/K}(0,|z|)\}
\end{equation}
where $\rho$ is the hyperbolic metric \cite[p. 19]{vu1}.
\end{theorem}
\begin{proof} Proof follows easily from
inequality (\ref{1ccc}) and the formula $\rho(0,r)= \log((1+r)/(1-r)) \, .$
\end{proof}

\begin{remark}\rm The constant $K$ can not be replaced by $K^{9/10}$ in (\ref{1dD}),
because for $|z|=t_0$, the inequality (\ref{1dD}) is equivalent to
$1-2\, \textrm{arth}(t^{1/K}_0)/K^{9/10}\geq 0$. Write
$f(K)=1-2\, \textrm{arth}(t^{1/K}_0)/K^{9/10}$, and we get
$$f^{'}(K)=\frac{9\,\textrm{arth}(t^{1/K}_0)}{5K^{19/10}}
+\frac{2t^{1/K}\log(t_0)}{K^{29/10}(1-t^{2/K}_0)},$$
we see that $f^{'}(1.005)=-0.004<0$, $f(K)$ is not increasing in $K$.
\end{remark}

\begin{lemma}\label{1D} For $K>1$ the function
$$F(r)=\frac{2{\rm arth}(1/{\rm ch}({\rm arch}(1/r)/K))}{\max\{2{\rm arth}(r),
(2{\rm arth}(r))^{1/K}\}}$$
is monotone increasing in $(0,t_0)$ and decreasing in $(t_0,1)$.
\end{lemma}

\begin{proof} (1) Let $u={\rm arch}(1/r)/K$ and
$$f(r)=\frac{{\rm arth}(1/{\rm ch}(u))}{{\rm arth}(r)}\,.$$
Differentiating $f$ with respect to $r$ we get
\begin{eqnarray*}
f^{'}(r)&=&-\frac{{\rm arth}(1/{\rm ch}(u))}{(1-r^2)({\rm arth}(r))^2}
+\frac{(1/{\rm ch}(u)){\rm th}(u)}{K\sqrt{1/r-1}\sqrt{1+1/r}\,r^2\,{\rm arth}(r)(1-(1/{\rm ch}(u))^2)}\\
&=&-\frac{Kr\,{\rm arth}(1/{\rm ch}(u)){\rm sh}(u)-\sqrt{1-r^2}\,{\rm arth}(r)}
{Kr(1-r^2)({\rm arth}(r))^2{\rm sh}(u)}\leq 0,
\end{eqnarray*}
by Lemma \ref{2e0}(2), hence $f$ is decreasing in $r\in(0,1)$.\\
(2) Let
$$g(r)=\frac{2^{1-1/K}{\rm arth}(1/{\rm ch}(u))}{({\rm arth}(r))^{1/K}}\,.$$
Differentiating $g$ with respect to $r$ we get
$$g^{'}(r)=\xi\left((1-r^2){\rm arth}(r)-
r\sqrt{1-r^2}\,{\rm arth}(1/{\rm ch}(u)){\rm sh}(u)\right)\geq 0$$
by Lemma \ref{2e0}(1), here
$$\xi=\frac{{2^{1-1/K}({\rm arth}(r))^{-(1+K)/K}}}{Kr(1-r^2)^{3/2}\,{\rm sh}(u)}\,.$$
Hence $g$ is increasing in $r\in(0,1)$.
We see that $f(t_0)=g(t_0)$. Thus $F(r)$ increases in $r\in(0,t_0)$
and decreases in $t\in(t_0,1)$.
\end{proof}

For instance it is well-known that for all $K>1, r \in (0,1)$
\begin{equation}
\log \left(\frac{1+ \varphi_K(r)}{1- \varphi_K(r)}\right) > \ K \log \left(\frac{1+r}{1-r}\right) \,
\end{equation}
\cite[(4.5)]{avvb1}. In the next theorem we study a function $p(r)$ which by
\cite[Thm 10.14]{avvb} is a minorant of $\varphi_K(r)\,.$

\begin{theorem}\label{1ddd} The following inequality holds for $K\geq 1$, $r\in(0,1),\,t_0=(e-1)/(e+1)$,
$$\log\left(\frac{1+p\,(r)}{1-p\,(r)}
\right)\leq c_3(K)\max\left\{\log\left(\frac{1+r}{1-r}\right),
\left(\log\left(\frac{1+r}{1-r}\right)\right)^{1/K}\right\}$$
here $p\,(r)=1/{\rm ch}({\rm arch}(1/r)/K)$ and $c_3(K)=2\,{\rm arth} (p\,(t_0))$.
Moreover, $c_3(K)\to 1$ when $K\to 1\,.$
\end{theorem}

\begin{proof}
\rm The inequality follows easily from Lemma \ref{1D}, because the maximum value of the function given in Lemma \ref{1D} is
$c_3(K)=1/{\rm ch}({\rm arch}(1/t_0)/K)$.   %\qquad $\square$
\end{proof}

We remark in passing that an inequality similar to (\ref{1ddd}) but with $p(r)$ replaced
with $\varphi_K(r)$ and $c_3(K)$ replaced with a constant $c(K)$ was proved in \cite[Lemma 4.8]{bv}.

%%%%%%%%%%%%%%%%
%%%%%%%%%%%%%%%%%%%
%%%%%%%%%%%%%%%%%%%%%%
\section{Quasiinvariance of the distance ratio metric}

Our goal in this section is to study how the distances in the $j$-metric
are transformed under the function (\ref{myfdef}) following
closely the paper \cite{kmv}. The main result here is Corollary \ref{jandmyf}.

%{\mathcal A}

\begin{lemma} The following inequality holds for $K\geq 1$:
\begin{equation}\label{2i}
\log\left(1+\frac{|{\mathcal A}_{1/K,K}(x)-{\mathcal A}_{1/K,K}(y)|}
{\min\{|{\mathcal A}_{1/K,K}(x)|,|{\mathcal A}_{1/K,K}(y)|\}}\right)
\leq 2^{1-1/K}\max\{\log^{1/K}(t),\log(t)\}
\end{equation}
here $t=1+\displaystyle\frac{|x-y|}{\min\{|x|,|y|\}}$, for all $x,y\in\mathbb{B}^n$.
\end{lemma}

\begin{proof} By Theorem \ref{Bthm} and Corollary \ref{2e}(1) we get
$$1+\frac{|{\mathcal A}_{1/K,K}(x)-{\mathcal A}_{1/K,K}(y)|}{\min\{|{\mathcal A}_{1/K,K}(x)|,|{\mathcal A}_{1/K,K}(y)|\}}\leq 1+2^{1-1/K}\frac{|x-y|^{1/K}}{\min\{|x|^{1/K},|y|^{1/K}\}}$$
$$\leq \left(1+\left(\frac{|x-y|}
{\min\{|x|,|y|\}}\right)^{1/K}\right)^{2^{1-1/K}}.\qquad (\ast)$$
Taking $\log$ both sides to $(\ast)$ and by Corollary \ref{2e}(3) we get
$$\log\left(1+\frac{|{\mathcal A}_{1/K,K}(x)-{\mathcal A}_{1/K,K}(y)|}{\min\{|{\mathcal A}_{1/K,K}(x)|,|{\mathcal A}_{1/K,K}(y)|\}}\right)
\leq \log\left(\left(1+\left(\frac{|x-y|}
{\min\{|x|,|y|\}}\right)^{1/K}\right)^{2^{1-1/K}}\right)$$
$$\leq2^{1-1/K}\max\left\{\log
\left(1+\frac{|x-y|}{\min\{|x|,|y|\}}\right),
\left(\log\left(1+\frac{|x-y|}{\min\{|x|,|y|\}}\right)\right)^{1/K}\right\}.$$
\end{proof}
We denote by $\partial G$ the boundary of a domain $G$ and define
$$d(z)=\min\{|z-m|:m\in\partial G\}.$$

For a domain  $G \subset \mathbb{R}^n,G\neq\mathbb{R}^n$, the following formula
$$j(x,y)=\log\left(1+\frac{|x-y|}{\min\{d(x),d(y)\}}\right),\;x,y\in G$$
defines $j$ as a metric in $G$ (see \cite[p.28]{vu1}).

\begin{corollary} {\label{jandmyf}}
Let $D=\mathbb{R}^n\setminus\{0\}$, then  we have
$$j_{D}({\mathcal A}_{1/K,K}(x),{\mathcal A}_{1/K,K}(y))\leq 2^{1-1/K}\max\{j_{D}(x,y),j^{1/K}_{D}(x,y)\}$$
for all $K \ge 1, x,y\in\mathbb{B}^n\cap D$.
\end{corollary}

\begin{proof}
Follows from inequality (\ref{2i}).
\end{proof}

%%%%%%%%%%%%
%%%%%%%%%%%%
%%%%%%%%%%%%
%%%%%%%%%%%%%
\section{Radial functions}
\begin{definition}\rm Let $f:\overline{\mathbb{R}}^n\to \overline{\mathbb{R}}^n$ be a homeomorphism. We say that $f$ is a \emph{radial function} if there exists a homeomorphism $g:(0,\infty)\to(0,\infty)$
such that $f(x)=g(|x|)x,\, x\in\mathbb{R}^n\setminus\{0\}$.

The following functions are examples of the radial functions:
\begin{enumerate}
\item $h(x)=\displaystyle\frac{x}{|x|^2}\,$, $x \in \mathbb{R}^n\setminus\{0\},$ $ \,h(0)=\infty$, $\,h(\infty)=0\,.$
\item For $a,b>0,$
$${\mathcal A}_{a,b}(x)=\left\{\begin{array}{ll}|x|^{a-1}x\quad if\; |x| \leq 1\\
                   |x|^{b-1}x\quad if\; |x|>1.\end{array}\right.$$
\end{enumerate}
\end{definition}

\begin{remark}\rm Properties of ${\mathcal A}:$
\begin{enumerate}

\item For $|x|<1$ and $a,b,c,d>0$
\begin{eqnarray*}
{\mathcal A}_{a,b}({\mathcal A}_{c,d}(x))&=&{\mathcal A}_{a,b}(|x|^{c-1}x)=||x|^{c-1}x|^{a-1}|x|^{c-1}x\\
                       &=&|x|^{ac-c}|x|^{c-1}x=|x|^{ac-1}x.
\end{eqnarray*}
\item For $|x|>1$
\begin{eqnarray*}
{\mathcal A}_{a,b}({\mathcal A}_{c,d}(x))&=&{\mathcal A}_{a,b}(|x|^{d-1}x)=||x|^{d-1}x|^{b-1}|x|^{d-1}x\\
                       &=&|x|^{bd-d}|x|^{d-1}x=|x|^{bd-1}x.
\end{eqnarray*}
(1) and (2) imply that ${\mathcal A}_{a,b}({\mathcal A}_{c,d}(x))={\mathcal A}_{ac,bd}(x)$.
\item ${\mathcal A}^{-1}_{a,b}(x)={\mathcal A}_{1/a,1/b}(x)$.
\end{enumerate}
\end{remark}

%%%%%%%%%%%%%%

%%%%%%%%%%%%%%%%%%%%

\begin{lemma}\cite[(1.5)]{vu1} An inversion in $S^{n-1}(a,r)$ is defined as,
$$h(x)=a+\frac{r^2(x-a)}{|x-a|^2},\; h(a)=\infty, \; h(\infty)=a.$$
Moreover,
\begin{equation}\label{2a}
|h(x)-h(y)|=\frac{r^2|x-y|}{|x-a||y-a|}.
\end{equation}
\end{lemma}

One of the goals of this section is to find a partial counterpart of the
distance formula (\ref{2a}) for $\mathcal A$ and to prove Theorem \ref{kal}.

\begin{lemma} Let $h(w)=r^2w/|w|^2,\,r>0,\,w\in\mathbb{R}^n\setminus\{0\}$ and let
$x,y\in\mathbb{R}^n\setminus\{0\}$ with $|x|\leq|y|$. Then with $\lambda=(|x|+|x-y|)/|x|$
and $z=\lambda x$ we have
$$\quad|h(x)-h(z)|\leq |h(x)-h(y)| \leq 3|h(x)-h(z)|.$$
Equality holds in the upper bound for $x=-y$.
\end{lemma}

\begin{proof} For the proof of first inequality we observe that
\begin{eqnarray*}
|h(x)-h(z)|&=&|h(x)-\frac{\lambda}{|\lambda|^2} h( x)|=\frac{|\lambda-1|}{\lambda}\frac{r^2}{|x|}\\
   &=&\frac{r^2|x-y|}{|x|(|x|+|x-y|)}\\
   &\leq&\frac{r^2|x-y|}{|x||y|}=|h(x)-h(y)|
\end{eqnarray*}
by triangle inequality.\\
For the second inequality, we have
\begin{eqnarray*}
\frac{|h(x)-h(y)|}{|h(x)-h(z)|}&=&\frac{|x-y|}{|x||y|}\frac{|x|(|x|+|x-y|)}{|x-y|}\\
                                &=&\frac{|x|}{|y|}+\frac{|x-y|}{|y|}\leq 1+\frac{|x|+|y|}{|y|}\leq 3.
\end{eqnarray*}
Note that here equality holds for $x=-y$.
\end{proof}

\begin{lemma}\label{2ii} The following inequality holds for $K\geq 1$:
$$||x|^{K-1}x-|y|^{K-1}y|\leq e^{\pi(K-1/K)}|x|^{K-1/K}\max\{|x-y|^{1/K},|x-y|^K\}$$
for all $x,y\in\mathbb{C}\setminus\overline{\mathbb{B}}^2$.
\end{lemma}

\begin{proof} By \cite[Theorem 14.18, (14.4)]{avvb} we get because
 $f:x\mapsto |x|^{K-1}x$ is $K$-quasiconformal \cite[16.2]{v}
$$||x|^{K-1}x,f(0),|y|^{K-1}y,f(\infty)|\leq \eta^{*}_{K,2}(|x,0,y,\infty|)=\eta_{K,2}\left(\frac{|x-y|}{|x|}\right).$$
Finally by \cite[Theorem 10.24]{avvb} and \cite[Remark 10.31]{vu1} we have
\begin{eqnarray*}
||x|^{K-1}x-|y|^{K-1}y|&\leq& |x|^K\eta_{K,2}\left(\frac{|x-y|}{|x|}\right)\\
                        &\leq &\lambda(K)|x|^K\max\{\left(\frac{|x-y|}{x}\right)^{1/K},
                        \left(\frac{|x-y|}{x}\right)^K\}\\
                        &\leq &e^{\pi(K-1/K)}|x|^{K-1/K}\max\{|x-y|^{1/K},|x-y|^K\}.
\end{eqnarray*}
\end{proof}

\begin{lemma}\label{2iii} The following inequality holds for $K\geq 1$ and for all $x,y\in\mathbb{R}^n\setminus\overline{\mathbb{B}}^n$:
$$||x|^{\beta-1}x-|y|^{\beta-1}y|\leq c(K)|x|^{\beta-\alpha}\max\{|x-y|^\alpha,|x-y|^\beta\}$$
here $c(K)=2^{K-1}K^K\exp(4K(K+1)\sqrt{K-1})$ and $\alpha=K^{1/(1-n)}=1/\beta$.
\end{lemma}

\begin{proof} By \cite[Theorem 14.18]{avvb} we get because
 $f:x\mapsto |x|^{K-1}x$ is $K$-quasiconformal \cite[16.2]{v}
$$||x|^{\beta-1}x,f(0),|y|^{\beta-1}y,f(\infty)|\leq \eta^{*}_{K,n}(|x,0,y,\infty|),$$
and this is equivalent to
$$||x|^{\beta-1}x-|y|^{\beta-1}y|\leq |x|^\beta\eta^{*}_{K,n}\left(\frac{|x-y|}{|x|}\right).$$
By \cite[Theorem 14.6]{avvb} we get
\begin{eqnarray*}
||x|^{\beta-1}x-|y|^{\beta-1}y|&\leq &c(K) |x|^\beta\max\left\{\left(\frac{|x-y|}{x}\right)^\alpha,
                        \left(\frac{|x-y|}{x}\right)^\beta\right\}\\
                        &\leq &c(K)|x|^{\beta-\alpha}\max\{|x-y|^\alpha,|x-y|^\beta\}.
\end{eqnarray*}
\end{proof}

\begin{corollary} The following inequalities hold for $K\geq 1$;
\begin{equation}\label{2j}
\left|\frac{x}{|x|^{1+1/K}}-\frac{y}{|y|^{1+1/K}}\right|\leq 2^{1-1/K}\frac{|x-y|^{1/K}}{(|x||y|)^{1/K}}
\end{equation}
for all $x,y\in \mathbb{R}^n\setminus\mathbb{B}^n$,

\begin{equation}\label{2k}
\left|\frac{x}{|x|^{1+\beta}}-\frac{y}{|y|^{1+\beta}}\right|\leq \frac{c(K)}{{|x|^{\beta-\alpha}}}
\max\left\{\left(\frac{|x-y|}{|x||y|}\right)^\alpha,\left(\frac{|x-y|}{|x||y|}\right)^\beta\right\}
\end{equation}
for all $x,y\in\mathbb{B}^n$,

\begin{equation}\label{2kk}
\left|\frac{x}{|x|^{1+K}}-\frac{y}{|y|^{1+K}}\right|\leq \frac{e^{\pi(K-1/K)}}{|x|^{K-1/K}}
\max\left\{\left(\frac{|x-y|}{|x||y|}\right)^{1/K},
\left(\frac{|x-y|}{|x||y|}\right)^K\right\}
\end{equation}
for all $x,y\in\mathbb{B}^2$.
\end{corollary}

\begin{proof} For the proof of (\ref{2j}) we define
$$g(z)={\mathcal A}_{1/K,K}(h(z))=\frac{z}{|z|^{1+1/K}}, \; h(z)=\frac{z}{|z|^2},\;z\in\mathbb{R}^n\setminus\mathbb{B}^n.$$
By Theorem \ref{Bthm} and (\ref{2a}) we get,
$$|g(x)-g(y)|=\left|\frac{x}{|x|^{1+1/K}}-\frac{y}{|y|^{1+1/K}}\right|\\
\leq 2^{1-1/K}|h(x)-h(y)|^{1/K} \leq 2^{1-1/K}\frac{|x-y|^{1/K}}{(|x||y|)^{1/K}}.$$
Again for the proof of (\ref{2k}) we define
$$g(z)={\mathcal A}_{\alpha,\beta}(h(z))=\frac{z}{|z|^{1+\beta}}, \; h(z)=\frac{z}{|z|^2},\;z\in\mathbb{B}^n.$$
By Lemma \ref{2iii} and  (\ref{2a}) we get,
\begin{eqnarray*}
|g(x)-g(y)|&\leq& c(K)|h(x)|^{\beta-\alpha}\max\{|h(x)-h(y)|^\alpha,|h(x)-h(y)|^\beta\}\\
           &=& \frac{c(K)}{|x|^{\beta-\alpha}} \max\left
           \{\left(\frac{|x-y|}{|x||y|}\right)^\alpha,
           \left(\frac{|x-y|}{|x||y|}\right)^\beta\right\}.
\end{eqnarray*}
Similarly, inequality (\ref{2kk}) follows from Lemma \ref{2ii} and  (\ref{2a}).
\end{proof}

%%%%%%%%%%%%%%%%%%%%%%%BEGIN CUT
\begin{lemma}\label{le} For $0< a\le 1\le p< \infty$ and $0\le s\le 2\pi$ we have
$$\frac{(1 + p^{2 a} - 2 p^a \cos s)^{1/2}}{(-1 + (1 + X)^a)}
\le \frac{1 + p^a}{(-1 + (2 + p)^a)}\,, \quad X = \sqrt{1 +
p^2 - 2 p \cos s}\, .$$
\end{lemma}

\begin{proof}
Let $$f_{p,a}(s)=\frac{(1 + p^{2 a} - 2 p^a \cos s)}{ (-1 +
(1 + X)^a)^2}.$$  Then
$$f_{p,a}'(s)=2\frac{(-a  (p^{1 - a} + p^{a + 1} - 2 p \cos s)/X +
    (1+X-(1+X)^{1-a})) \sin s}{
 p^{-a} (1 + X)^{1 - a}  (-1 + (1 + X)^a)^3}
$$

As $$p^{1 - a} + p^{a + 1}\le 1+p^2$$ because
$$p^{1+a}(1-p^{1-a})\le 1-p^{1-a}$$ it follows that
$$f'_{p,a}(s) / \sin s\ge 2\frac{(-a  X +
    (1 +  X -  (1 + X)^{1 - a}))}{
 p^{-a} (1 + X)^{1 - a} (-1 + (1 + X)^a)^3}.
$$
As $$(1+X)^{1-a}< 1+(1-a)X,$$ it follows that
$$f'_{p,a}(s)=0\text{ if and only $s=0$ or $s=\pi$.}$$

For $s=0$,  the function $f_{p,a}$ achieves its minimum $$
f_{p,a}(0)=\left(\frac{-1 + p^
   a}{-1 + {p}^a}\right)^2=1$$ and for $s=\pi$ its maximum $$
f_{p,a}(\pi)=\left(\frac{1 +p^
   a}{(-1 + (2 + p)^a)}\right)^2.$$
\end{proof}

\begin{lemma}\label{ve}
For $p\ge 1$, and $0<d\le 1$ there holds
\begin{equation}\label{bif} \frac{1+p^d}{(2+p)^d-1}\le
\frac{2}{3^d-1}.
\end{equation}
\end{lemma}

\begin{proof}
Let $$h(p)=(3^d-1)(1+p^d)-2((2+p)^d-1).$$ We need to show
that $h(p)\le 0$. First of all
$$h'(p)=d\left((3^d-1)p^{d-1}-2(2+p)^{d-1}\right).$$ Then
$$h'(p)\le 0 \Leftrightarrow \left(\frac{2}{p}+1\right)^{1-d}\le \frac{2}{3^d-1}.$$
Since $$\left(\frac{2}{p}+1\right)^{1-d}\le 3^{1-d}$$ we need to
show that $$\left(3\right)^{1-d}\le \frac{2}{3^d-1},$$ but this is
equivalent to $$3^d\le 3$$ which is obviously true. Thus
$h'(p)\le 0$, and consequently $h(p)\le h(1)=0$ and this
inequality coincides with \eqref{bif}.
\end{proof}

\begin{subsec}{\bf Proof of  Theorem \ref{kal}.}  \rm
{\it The case $1\le |x|\le |y|$}. Let us show that $Q(x,y)\le
1$.
Without loss of generality, we can assume that $x=r$ and $z$ are
positive real numbers, and $y=R e^{it}$. Then $z=r + |r-Re^{it}|$.
Let $$p = \frac{R}{r}.$$ Then $p\ge 1$. Next we have:
\[\begin{split}\frac{|{\mathcal A}(x)-{\mathcal A}(y)|}{|{\mathcal A}(x)-{\mathcal A}(z)|}&= \frac{|1-p^b
e^{it}|}{(1+|1-p e^{it}|)^b-1}\\&\le \frac{|1-p^b
e^{it}|}{(1+|1-p|)^{b-1}(1+|1-p e^{it}|)-1}\\&=\frac{|1-p^b
e^{it}|}{p^{b-1}(1+|1-p e^{it}|)-1}\\&=\frac{|1-p^b
e^{it}|}{p^{b-1}-1+|p^{b-1}-p^b
e^{it}|}\\&=\frac{|1-p^{b-1}+p^{b-1}-p^b
e^{it}|}{p^{b-1}-1+|p^{b-1}-p^b e^{it}|}\le 1.\end{split}\]

If $|x|\leq |y|\leq 1$ and $|z|\le 1$, then by Lemmas \ref{le} and \ref{ve} we get
\begin{eqnarray*}
\frac{|{\mathcal A}(x)-{\mathcal A}(y)|}{|{\mathcal A}(x)-{\mathcal A}(z)|}&=&\frac{|r^a-R^a
e^{it}|}{(r+|r-Re^{it}|)^a-r^a}\\
&\leq &\frac{1+p^a}{(2+p)^a-1}
 \leq \frac{2}{3^a-1}.
 \end{eqnarray*}

If $|x|\le |y|\le 1$  and $|z|\ge 1$, then it
follows  from Lemmas \ref{le} and \ref{ve} and  $|z|^b\ge |z|^a$  that
\begin{eqnarray*}
\frac{|{\mathcal A}(x)-{\mathcal A}(y)|}{|{\mathcal A}(x)-{\mathcal A}(z)|}&\leq &
\frac{|r^a-R^a e^{it}|}{(r+|r-Re^{it}|)^a-r^a}\\
&\leq &\frac{1+p^a}{(2+p)^a-1}
\leq \frac{2}{3^a-1}.
 \end{eqnarray*}

{\it The case $|x|\le 1\le |y|$ and $r^{a-1}> R^{b-1}$.} Then there
holds
$$Q(x,y) \le \frac{2}{3^a-1}.$$
First of all
\[\begin{split}\frac{ |{\mathcal A}(x)-{\mathcal A}(y) |}{ |{\mathcal A}(x)-{\mathcal A}(z) |}&=\frac{|r^a-R^b e^{it}|}{(r+|r-Re^{it}|)^b-r^a}
\\&=\frac{|\alpha-e^{it}|}{(\beta+|\beta-e^{it}|)^b-\alpha}
\end{split}
\]
where $\alpha=\frac{r^a}{R^b}$ and $\beta = \frac r R$. Take the
continuous function $k(q) = \beta^q$, $a\le q\le 1$. Since
$$\beta=k(1)=\frac{r}{R}\le \alpha=\frac{r^a}{R^b}\le  k(a)= \frac{r^a}{R^a} $$ it follows
that there exists a constant $c$ with $ a\le c \le 1$ such that $k(c)=\beta^c =
\alpha$. Then
\begin{eqnarray*}
\frac{|\alpha-Re^{it}|}{(\beta+|\beta-e^{it}|)^b-\alpha}&=&
\frac{|\beta^c-e^{it}|}{(\beta+|\beta-e^{it}|)^b-\beta^c}\\
&\leq &\frac{|\beta^c-e^{it}|}{(\beta+|\beta-e^{it}|)^c-\beta^c}
\\
&\leq & \frac{1+\beta^c}{(2+\beta)^c-\beta^c}\\
 &\leq & \frac{2}{3^c-1}
 \leq \frac{2}{3^a-1},
 \end{eqnarray*}
the second inequality follows from Lemma \ref{le} and the third inequality follows from Lemma
\ref{ve} by taking $p=1/\beta$ and $c=d$.

Finally, let us show that $C(a,b)\ge 2/(3^a-1)\,.$ Suppose that $x \in
{\mathbb R}^n \setminus \{ 0 \} $ is such that $3 |x|<1\,,$ i.e.
$0 < |x|<1/3\,$ and $y=-x\,.$ Then $z= x(|x|+|x-y|)/|x| = 3 x$ and
$$
Q(x, -x)= \frac{2 |x|^a}{ (3|x|)^a - |x|^a}=  \frac{2 }{3^a-1} \,,
$$
and hence $C(a,b)\ge 2/(3^a-1)\,.$
$\square$
\end{subsec}

\section{Conclusions}

%For the comparision of our results with Theorems \ref{Mthm} and \ref{Dthm} we need the
%following result.

\begin{subsec}{\bf Proof of  Theorem \ref{mymaj}. }\rm
{} For $|x|,|y|<1$ we have
$$\alpha_p(x,y)=\left||x|^{p-1}x-|y|^{p-1}y\right|=
\left|{\mathcal A}_{p, 1/p}(x) - {\mathcal A}_{p, 1/p}(y)  \right|.$$
Consider the case $|x|<1<|y|\,$. It is obvious that
$$\cos\theta\leq 1<\frac{|x|^{-p}(|y|^{1/p}+|y|^{p})}{2},$$
this is equivalent to
$$\cos\theta\leq 1< \frac{(|y|^{1/p}-|y|^{p})(|y|^{1/p}+|y|^{p})}{2|x|^p(|y|^{1/p}-|y|^{p})}$$
$\Longleftrightarrow$
$$\qquad  2|x|^p|y|^{1/p}\cos\theta-2|x|^p|y|^{p}\cos\theta < |y|^{2/p}-|y|^{2p}$$
$\Longleftrightarrow$
$$|y|^{2p}-2|x|^{p-1}|y|^{p-1}|x||y|\cos\theta <
|y|^{2/p}-2|x|^{p-1}|y|^{1/p-1}|x||y|\cos\theta$$
$\Longleftrightarrow$
$$||x|^{p-1}x|^2+||y|^{p-1}y|^2-2|x|^{p-1}|y|^{p-1}x\,y
< ||x|^{p-1}x|^2+||y|^{1/p-1}y|^2-2|x|^{p-1}|y|^{1/p-1}x\,y$$
$\Longleftrightarrow$
$$\left||x|^{p-1}x-|y|^{p-1}y\right|^2<\left||x|^{p-1}x-|y|^{1/p-1}y\right|^2
=\left|{\mathcal A}_{p, 1/p}(x) - {\mathcal A}_{p, 1/p}(y)\right|^2.$$

%%%%%%%%%%NEW STARTS HERE

Consider now the case $1<|x|<|y|$. Starting with the observation that the function $t
\mapsto t^{1/p}-t^p $ is increasing for $t>1$ when $p \in (0,1)\,,$
we see that
$$\frac{|x|^{1/p}}{|x|^p}\left(\left(\frac{|y|}{|x|}\right)^{1/p}-1\right)
>\left(\left(\frac{|y|}{|x|}\right)^{p}-1\right)\Leftrightarrow
(|y|^{1/p}-|x|^{1/p})^2>(|y|^{p}-|x|^{p})^2$$ $\Longleftrightarrow$
$$|x|^{2/p}-|y|^{2p}+|y|^{2/p}-|x|^{2p}>2|x|^{1/p}|y|^{1/p}-2|x|^{p}|y|^{p}.$$
Now it is clear that
$$\cos\theta\leq 1<\frac{|x|^{2/p}-|y|^{2p}+|y|^{2/p}-|x|^{2p}}
{2|x|^{1/p}|y|^{1/p}-2|x|^{p}|y|^{p}}$$ $\Longleftrightarrow$
$$|x|^{2p}+|y|^{2p}-2|x|^{p}|y|^{p}\cos\theta
<|x|^{2/p}+|y|^{2/p}-2|x|^{1/p}|y|^{1/p}\cos\theta$$

$\Longleftrightarrow$
$$||x|^{p-1}x|^2+||y|^{p-1}y|^2-2|x|^{p-1}|y|^{p-1}x\,y
< ||x|^{1/p-1}x|^2+||y|^{1/p-1}y|^2-2|x|^{1/p-1}|y|^{1/p-1}x\,y$$

$\Longleftrightarrow$
$$||x|^{p-1}x|^2+||y|^{p-1}y|^2-2|x|^{p-1}|y|^{p-1}x\,y
<||x|^{p-1}x|^2+||y|^{1/p-1}y|^2-2|x|^{p-1}|y|^{1/p-1}x\,y$$

$\Longleftrightarrow$
$$\left||x|^{p-1}x-|y|^{p-1}y\right|^2<\left||x|^{1/p-1}x-|y|^{1/p-1}y\right|^2
=\left|{\mathcal A}_{p, 1/p}(x) - {\mathcal A}_{p, 1/p}(y)\right|^2.  \quad \square$$
\end{subsec}

%{\Large \bf Some comparisons will be inserted here.}

%\bigskip

\begin{subsec} {\rm
{\bf Comparison of the bounds.}
In what follows, we use the symbols $M,D, B,K$ for the bounds given by Theorems
\ref{Mthm}, \ref{Dthm}, \ref{Bthm}, \ref{kal}, respectively. In the case of the
complex plane, we will show by numerical examples that each of these four bounds
can occur as minimal. To this end, for each of the symbols $M,D, B,K$,
we give a table of four $x,y$ pairs and the corresponding upper bound values associated
with the four symbols  $M,D, B,K$, such that the bound associated with the symbol
in question is the least one. For the computation of the $K$ bound it should be observed
that in Theorem \ref{mymaj} we have the constraint $|x|\le|y|\,.$ If this is not the
situation to begin with, we have swapped the points for computation. In Tables 1-4
the parameter $p=0.5\,.$

\begin{table}[ht]
\caption{Sample points with $K<\min \{B,D,M\} \,.$ }\label{table1}
\renewcommand\arraystretch{1}
\noindent
\begin{displaymath}
\begin{array}{|c|c|c|c|c|c|c|}
\hline
k&x_k&y_k&B&D&M&K\\
\hline

1&-2.00-2.65i & 2.65-2.65i & 3.0496 & 143.4290 & 3.6030 & 2.6591\\

2&2.25-0.75i &2.65+1.30i & 2.0438 & 38.9860&1.8236 & 1.5158\\

3&1.35+0.50i &1.95-0.65i & 1.6107 & 14.8000&1.3571&1.2768\\

4&1.10+2.30i &-2.40+2.10i & 2.6479 & 82.4142 &2.9447&2.3646\\

\hline
\end{array}
\end{displaymath}
\end{table}

%
%\begin{figure}[!ht]
%\includegraphics[width=8cm]{pK.eps}
%\caption{ $K<\min\{B,D,M\}$.}
%\end{figure}
%
%
%\pagebreak

\begin{table}[ht]
\caption{Sample points with $D<\min \{B,K,M\} \,.$ }\label{table2}
\renewcommand\arraystretch{1}
\noindent
\begin{displaymath}
\begin{array}{|c|c|c|c|c|c|c|}
\hline
k&x_k&y_k&B&K&M&D\\
\hline

1&0.80-0.50i & -1.80+1.45i & 3.6968 & 45.3884 & 3.2066 & 2.5495\\

2&2.25-0.75i &0.00-0.05i & 15.5147     & 32.3855&2.7931 & 2.6174\\

3&2.55+1.50i&-1.10+1.70i & 2.8148   & 76.9511&3.1879&2.7039\\

4&-2.70+3.00i&1.50+0.60i& 4.2727  &106.6320 &3.6118&3.1104\\

\hline
\end{array}
\end{displaymath}
\end{table}

%\begin{figure}[!ht]
%\includegraphics[width=8cm]{pD.eps}
%\caption{$D<\min\{B,K,M\}$.}
%\end{figure}
%
%
%
%%Here $B<\min \{D,K,M \}\,.$
%
%\pagebreak

\begin{table}[ht]
\caption{Sample points with $B<\min \{D,K,M\} \,.$ }\label{table3}
\renewcommand\arraystretch{1}
\noindent
\begin{displaymath}
\begin{array}{|c|c|c|c|c|c|c|}
\hline
k&x_k&y_k&D&K&M&B\\
\hline
1&-2.45-2.205i & -1.2+0.55i & 2.92 & 43.55 & 2.42 & 2.40\\
2&-1.65+1.45i &2.15+2.75i & 3.01& 92.27&3.22 & 2.83\\
3&-0.2-3i&-0.4+0.2i & 5.21   & 34.64&2.77&2.53\\

4&0.9-2.9i&-1.4+1.35i& 3.74 &115.15 &4.16&3.11\\

\hline
\end{array}
\end{displaymath}
\end{table}

%\begin{figure}[!hb]
%\includegraphics[width=8cm]{pB.eps}
%\caption{ $B<\min\{D,K,M\}$.}
%\end{figure}
%
%
%\pagebreak
%Finally, here $M<\min\{B,D,K\}$

\begin{table}[ht]
\caption{Sample points with $M<\min \{B,D,K\} \,.$ }\label{table4}
\renewcommand\arraystretch{1}
\noindent
\begin{displaymath}
\begin{array}{|c|c|c|c|c|c|c|}
\hline
k&x_k&y_k&B&D&K&M\\
\hline

1&0.30+0.50i & -0.15+2.95i & 2.23 & 3.69 & 23.73 & 2.17\\

2&0.95+1.85i & 0.55+1.55i & 1.00 & 0.53 & 5.18 & 0.52\\

3&1.60-0.25i & 1.10-0.35i & 1.01 & 0.64 & 3.93 & 0.60\\

4&-0.60+0.30 & -3.00+1.95i & 2.41 & 4.02 & 32.84 & 2.31\\

\hline
\end{array}
\end{displaymath}
\end{table}

%\begin{figure}[!ht]
%\includegraphics[width=8cm]{pM.eps}
%\caption{ $M<\min\{B,D,K\}$.}
%\end{figure}
%
%\pagebreak

} % end of \rm
\end{subsec}

\begin{table}[ht]
\caption{Sample points with $M<\min\{(\ref{2j}),D\} \,.$ }\label{table5}
\renewcommand\arraystretch{1}
\noindent
\begin{displaymath}
\begin{array}{|c|c|c|c|c|c|}
\hline
k&x_k&y_k&(\ref{2j})&D&M\\
\hline

1&2.25+2.45i & -0.01+2.95i & 0.27 & 0.27 & 0.24\\

2&-2.60+0.40i & -0.70-0.60i & 1.23 & 3.30 & 1.19\\

3&0.75-0.75i & -2.90-2.50i & 1.32 & 4.53 & 1.23\\

4&2.90+1.90i & 1.20+0.85i & 0.75 & 1.67 & 0.71\\

\hline
\end{array}
\end{displaymath}
\end{table}

%\begin{figure}[!ht]
%\includegraphics[width=8cm]{pMM1.eps}
%\caption{ $M<\min\{(\ref{2j}),D\}$.}
%\end{figure}
%
%
%\pagebreak

\begin{table}[ht]
\caption{Sample points with $(\ref{2j})<\min\{D,M\} \,.$ }\label{table6}
\renewcommand\arraystretch{1}
\noindent
\begin{displaymath}
\begin{array}{|c|c|c|c|c|c|}
\hline
k&x_k&y_k&D&M&(\ref{2j})\\
\hline

1&-2.60-1.05i & -1.35-1.40i & 0.70 & 0.65 & 0.56\\

2&-0.45-1.05i & 2.35+1.80i & 3.95 & 1.83 & 1.46\\

3&-1.15+2.30i & 2.70+0.65i & 0.99 & 2.12 & 0.96\\

4&-0.10+1.25i & 2.90+2.45i & 0.71 & 0.94 & 0.60\\

\hline
\end{array}
\end{displaymath}
\end{table}

%\begin{figure}[!ht]
%\includegraphics[width=8cm]{p8.eps}
%\caption{ $(\ref{2j})<\min\{D,M\}$.}
%\end{figure}
%
%\pagebreak

\begin{table}[ht]
\caption{Sample points with $D<\min\{(\ref{2j}),M\} \,.$ }\label{table7}
\renewcommand\arraystretch{1}
\noindent
\begin{displaymath}
\begin{array}{|c|c|c|c|c|c|}
\hline
k&x_k&y_k&(\ref{2j})&M&D\\
\hline

1&1.35+2.95i & -1.35+2.90i & 0.59 & 1.07 & 0.43\\

2&-0.80+2.75i & -1.85+2.40i & 0.38 & 0.49 &  0.25\\

3&2.65+2.20i & -2.45+2.40i & 0.49 & 0.64 & 0.49\\

4&1.20-0.70i & 1.30+0.70i & 1.05 & 1.96 & 0.91\\

\hline
\end{array}
\end{displaymath}
\end{table}

In conclusion, Tables 1-4 demonstrate that each of the above four bounds is sometimes
smaller than the minimum of the other three bounds. Some further results, in addition
to Theorems \ref{Mthm}, \ref{Dthm}, \ref{Bthm}, \ref{kal} can be found
in the papers \cite{m} and \cite{d}. The tables were compiled with the help of the
Mathematica software package.
%%%%%%%%%%%%%%%%%%%%

In Tables 5-7 we compare $(\ref{2j}), M$ and $D$, for $x,y\in \mathbb{R}^n\setminus
\mathbb{B}^n,\, p=-0.6.$

%\begin{figure}[!ht]
%\includegraphics[width=8cm]{pDD.eps}
%\caption{ $D<\min\{(\ref{2j}),M\}$.}
%\end{figure}
%\pagebreak

\small
%\begin{verbatim}

\bigskip

\noindent
{\sc
B. A. Bhayo and M. Vuorinen}\\
Department of Mathematics\\
University of Turku\\
20014 Turku\\
Finland\\
barbha@utu.fi\\

\noindent
{\sc V. Bo\v zin}\\
Faculty of Mathematics\\
University of Belgrade\\
Studentski trg 16, Belgrade\\
Serbia\\
bozinv@turing.mi.sanu.ac.rs\\

\noindent
{\sc
D. Kalaj}\\
University of Montenegro\\
Faculty of Mathematics\\
Dzordza Vašingtona b.b.\\
Podgorica\\
Montenegro\\
davidk@ac.me\\

%\end{verbatim}

\end{document}